\documentclass[10pt]{amsart}

\title{On the Hausdorff dimension of $\textrm{CAT}(\kappa)$ surfaces}

\author{David Constantine}
\address{
Wesleyan University \\
Mathematics and Computer Science Department \\
Middletown, CT 06459}

\author{Jean-Fran\c{c}ois Lafont}
\address{Department of Mathematics\\
                 Ohio State University\\
                 Columbus, Ohio 43210}

\date{\today}

\newtheorem{thm}{Theorem}
\newtheorem{lem}[thm]{Lemma}
\newtheorem{prop}[thm]{Proposition}
\newtheorem{cor}[thm]{Corollary}

\theoremstyle{definition}
\newtheorem*{rem}{Remark}

\newtheorem{ques}{Question}
\newtheorem{defn}[thm]{Definition}

\numberwithin{equation}{section}

\def\Pb{\ifmmode{\Bbb P}\else{$\Bbb P$}\fi}
\def\Z{\ifmmode{\Bbb Z}\else{$\Bbb Z$}\fi}
\def\Q{\ifmmode{\Bbb Q}\else{$\Bbb Q$}\fi}
\def\C{\ifmmode{\Bbb C}\else{$\Bbb C$}\fi}
\def\R{\ifmmode{\Bbb R}\else{$\Bbb R$}\fi}
\def\H{\ifmmode{\Bbb H}\else{$\Bbb H$}\fi}

\def\diam{\operatorname{diam}}

\usepackage{fancyhdr}
\usepackage{graphicx} 
\usepackage{amsmath,amssymb,amsthm,amsfonts,amscd,flafter,epsf,bm}
\usepackage{mathtools}
\usepackage{graphics}
\usepackage{MnSymbol}
\usepackage{epsfig}
\usepackage{psfrag}
\usepackage{color}
\usepackage{xy}
\input xy
\xyoption{all}

\begin{document}
\maketitle

\begin{abstract}

We prove a that a closed surface with a $\textrm{CAT}(\kappa)$ metric has Hausdorff dimension $=2$, and that there are uniform upper and lower bounds on the two-dimensional Hausdorff measure of small metric balls. We also discuss a connection between this uniformity condition and some results on the dynamics of the geodesic flow for such surfaces. Finally, we give a short proof of topological entropy rigidity for geodesic flow on certain $\textrm{CAT}(-1)$ manifolds.

\end{abstract}

\maketitle

\setcounter{secnumdepth}{1}

\setcounter{section}{0}

%
\section{Introduction}


Let $\Sigma$ be a closed surface, and let $d$ be a locally $\textrm{CAT}(\kappa)$ metric on $\Sigma$. One quantity of natural interest is the Hausdorff dimension of $(\Sigma,d)$, denoted $\dim_H(\Sigma,d)$. This dimension is bounded below by 2, the topological (covering) dimension of $\Sigma$. However,  for an arbitrary metric on $\Sigma$ there is no upper bound; this can be seen by `snowflaking' the metric -- replacing $d(x,y)$ with $d'(x,y)=d(x,y)^\alpha$ for $0<\alpha<1$, which raises the the dimension by a factor of $1/\alpha$ (see, e.g. \cite{tyson_wu}). In this paper we examine the restriction placed on $\dim_H(\Sigma,d)$ by the $\textrm{CAT}(\kappa)$ condition, and prove the following theorem:

\begin{thm}\label{main thm}
Let $(\Sigma,d)$ be a $\textrm{CAT}(\kappa)$ closed surface. Then $\dim_H(\Sigma,d)=2$. Moreover, there exists some $\delta_0>0$ such that for all $0<\delta\leq \delta_0$,
\[ \inf_{p\in\Sigma}H^2(B(p,\delta))>0  \ \ \mbox{ and } \ \ \sup_{p\in\Sigma}H^2(B(p,\delta))<\infty \]
where $H^2$ denotes the 2-dimensional Hausdorff measure and $B(p,\delta)$ is the ball of radius $\delta$ around $p$.
\end{thm}

We note that the second statement of the theorem implies the first, but not vice versa. Indeed there are metric spaces with Hausdorff dimension $d$ whose $d$-dimensional Hausdorff measures are zero or infinite.

We became interested in this question for Hausdorff measures, in particular the uniform bounds on the measures of balls, while thinking about some results on entropy for geodesic flows on $\textrm{CAT}(-1)$ manifolds. We include a discussion of these results in the final section of the paper, but here we state the two main theorems we prove.

\begin{thm}\label{thm:surface vol entropy}
Let $(\Sigma,d)$ be a closed surface with a $\textrm{CAT}(0)$ metric. Let $\phi_t$ be the geodesic flow on the space of geodesics for $(\Sigma,d)$. Then the topological entropy for the flow equals the volume growth entropy for the Hausdorff 2-measure.
\end{thm}

\begin{thm}\label{entropy rigidity}
Let $(X,d)$ be a closed $\textrm{CAT}(-1)$ manifold (not necessarily Riemannian), and suppose that $X$ admits a Riemannian metric $g$ so that $(X,g)$ is a locally symmetric space. Let $h_{top}(\phi^d_t)$ and $h_{top}(\phi^g_t)$ be the topological entropies for the geodesic flows under the two metrics. Then
\[ h_{top}(\phi^d_t)\geq h_{top}(\phi^g_t) \]
and if equality holds, $(X,d)$ is also locally symmetric. If $\dim X>2$, $(X,d)$ and $(X,g)$ are isometric.
\end{thm}

Theorem \ref{thm:surface vol entropy} is a version of Manning's \cite{manning} analogous result for Riemannian manifolds of non-positive curvature, and relies on some work of Leuzinger \cite{leuzinger}.  Theorem \ref{entropy rigidity} follows from our Theorem \ref{thm:surface vol entropy} and a rigidity result of Bourdon (\cite{bourdon_cr}). Our main work is to note how, via Theorem \ref{thm:surface vol entropy}, Bourdon's theorem can be recast as a topological entropy rigidity theorem. This fact may well be known to experts, but we have not found it addressed in the literature.

The paper is organized as follows. In Sections \ref{sec:circle} and \ref{sec:rectifiable} we show that small distance spheres around each point in $\Sigma$ are topological circles, and that they are rectifiable with bounded length. In Section \ref{sec:proof} we prove Theorem \ref{main thm}, and in Section \ref{sec:example} we discuss the extension of Main Theorem to higher dimensions, and give an example which indicates some of the complications in doing so. In Section \ref{sec:entropy} we give the proof of topological entropy rigidity (Theorem \ref{entropy rigidity}) for $\textrm{CAT}(-1)$ manifolds.

\

\subsection{Acknowledgements}

We would like to thank Enrico Leuzinger and Mike Davis for helpful conversations. The first author would like to thank Ohio State University for hosting him during the time that most of this work was done. The second author was partially supported by the NSF, under grant DMS-1510640.

\

%

\section{The topology of small distance spheres}\label{sec:circle}

Let $S(p, \epsilon) = \{z\in \Sigma: d(p,z)=\epsilon\}$ and $B(p, \epsilon) = \{z\in \Sigma: d(p,z)\leq \epsilon\}$ respectively denote the metric $\epsilon$-sphere and $\epsilon$-ball centered at $p$. In this section we prove for small $\epsilon$, all $S(p,\epsilon)$ are topological circles. We note that the argument only works for surfaces. In Section \ref{sec:example} we give examples of higher-dimensional $\textrm{CAT}(-1)$ manifolds where the analogous statement is not true.

Throughout this section, we work at small scale. We fix $\epsilon_0>0$ small enough so that the following two conditions are satisfied:
\begin{itemize}
	\item $\epsilon_0\leq D_\kappa/2$ where $D_\kappa$ is the diameter of the model space of constant curvature $\kappa$, and

	\item  For all $p\in\Sigma$, $B(p,\epsilon_0)$ is (globally) $\textrm{CAT}(\kappa)$.
\end{itemize}
At these scales, $B(p,\epsilon_0)$ is locally uniquely geodesic -- in particular there is a unique geodesic from $p$ to any point in $B(p,\epsilon_0)$, which varies continuously with respect to the endpoints. This will be a key fact in the work below. As a consequence, each such ball $B(p, \epsilon_0)$ is contractible, hence lifts isometrically to the universal cover $(\tilde \Sigma, \tilde d)$.

The following Lemma will be useful. Its proof, which is straightforward and can be adapted to any dimension, can be found in \cite[Proposition II.5.12]{bh}.

\begin{lem}\label{lem:extendible}
Let $[xy]$ be a geodesic segment in $\Sigma$ connecting an arbitrary pair of points $x$ and $y$. Then $[xy]$ can be extended beyond $y$. That is, there is a geodesic segment (not necessarily unique) $[xy']$ properly containing $[xy]$ as its initial segment. 
\end{lem}

\begin{rem}
Using the compactness of $X$, and a connectedness argument on $\mathbb{R}$, this lemma 
implies that each geodesic segment $[xy]$ can be infinitely extended.
\end{rem}




The main result of this section is the following:

\begin{prop}\label{prop:circle}
Let $\Sigma$ be a complete $\textrm{CAT}(\kappa)$ surface. Then for all $\epsilon<\epsilon_0$, $S(p,\epsilon)$ is homeomorphic to the circle $\mathbb{S}^1$.
\end{prop}

In order to establish this result, we use a well-known characterization of the circle $\mathbb{S}^1$.
The circle is the only compact, connected, metric space $(X,d)$ with the property that for any pair
of distinct points $a,b \in X$, the complement $X\setminus \{a,b\}$ is disconnected (see, e.g \cite[Theorem 2-28]{hocking_young}). 
Let $p\in \Sigma$ be an arbitrary point in $\Sigma$, and to simplify notation, we set $S_\epsilon := S(p, \epsilon)$.
We now claim that for $\epsilon<\epsilon_0$, $S_\epsilon$ is homeomorphic to a circle.

\begin{lem}\label{lem:pc}
For all $\epsilon<\epsilon_0$, $S_\epsilon$ is compact, path-connected, metric space.
\end{lem}

\begin{proof}
$S_\epsilon$ is a closed subset of the compact metric space $\Sigma$, so it is compact and metric. Since $\epsilon<\epsilon_0$, $S_\epsilon$ lifts homeomorphically to a subset of $\tilde \Sigma$. Since $\Sigma$ is a surface, $\tilde \Sigma$ is homeomorphic to $\mathbb{R}^2$ or $\mathbb{S}^2$, so we may take $S_\epsilon$ to a be a compact subset of $\mathbb{R}^2$ or $\mathbb{S}^2$.

$S_\epsilon$ has diameter $<D_\kappa$, so we may find a path $S$ in $\mathbb{R}^2$ or $\mathbb{S}^2$ homeomorphic to $\mathbb{S}^1$ bounding a disk containing $S_\epsilon$ and remaining in $B(p,\epsilon_0)$. Let $proj:S \to S_\epsilon$ be the nearest point projection (for the $\textrm{CAT}(\kappa)$ metric lifted from $\Sigma$). This is a continuous map. Since geodesics in $\tilde \Sigma$ are infinitely extendible, for any point $z$ on $S_\epsilon$, the geodesic segment $[pz]$ extends to a geodesic which hits $S$ at a point $q$. Then $proj(q)=z$ and so this map is also surjective. The surjective, continuous map from the path-connected set $S$ to $S_\epsilon$ proves that the latter is path-connected.
\end{proof}

\begin{lem}\label{lem:pairs-separate}
For all $\epsilon<\epsilon_0$, and pairs of distinct points $\{a, b\} \subset S_\epsilon$, the space
$S_\epsilon \setminus \{a, b\}$ is disconnected.
\end{lem}

\begin{proof}
As in the proof of Lemma \ref{lem:pc}, let $S$ in $\mathbb{R}^2$ or $\mathbb{S}^2$ be a Jordan curve containing $S_\epsilon$ and remaining in $B(p,\epsilon_0)$. In particular, from the Schoenflies 
theorem, we can now view $S_\epsilon$ as contained inside a closed topological disk $\mathbb D^2$ (the curve $S$ along with its interior).
Given the two distinct points $a, b\in S_\epsilon$, extend the two geodesic segments $[pa], [pb]$
to geodesic segments $[pa'], [pb']$, where $a', b'$ are a pair of (distinct) points on $S$, and 
$[pa), [pb)$ are both contained in the interior of $\mathbb D^2$.

Note that $[pa'], [pb']$ might coincide on some subgeodesic originating from $p$. Let $w$ be the
point at which these two geodesic segments separate, and consider the concatenation $[a'w]$
with $[wb']$. By uniqueness of geodesics in $B(p,\epsilon_0)$, we have that
$[a'w]\cap [wb'] = \{w\}$, and hence they concatenate to give an embedded arc joining the pair
of points $a', b' \in S = \partial \mathbb D^2$. Moreover, the interior $(a'b')$ of the arc is contained
in the interior of $\mathbb D^2$. In follows that $(a'b')$ separates the interior of the disk into two
connected components $U_1, U_2$. (See Figure \ref{fig:circle}.)


\begin{center}

\setlength{\unitlength}{.5pt}

\begin{picture}(600,500)(-300,-200)

\put(0,-100){\circle*{5}}
\put(0,-185){\circle*{5}}

\put(-15, -200){$p$}

\put(0,-185){\line(0,1){85}}

\put(-25,-105){$w$}

\qbezier(-200,0)(0,170)(200,0)

\put(-200,35){$S_\epsilon$}

\put(0,-100){\line(-1,2){200}}
\put(0,-100){\line(1,2){180}}

\put(-85,70){\circle*{5}}
\put(85,70){\circle*{5}}

\put(-100,45){$a$}
\put(90,45){$b$}

\qbezier(-200,250)(0,300)(-20,250)
\qbezier(-20,250)(-60,200)(200,200)

\put(-178,257){\circle*{5}}
\put(150,201){\circle*{5}}

\put(-195,230){$a'$}
\put(152,175){$b'$}

\put(190,210){$S$}

\put(-180,140){$[a'w]$}
\put(130,140){$[wb']$}

\put(-10,140){$U_1$}
\put(150,-100){$U_2$}

\end{picture}

\begin{figure}[h]
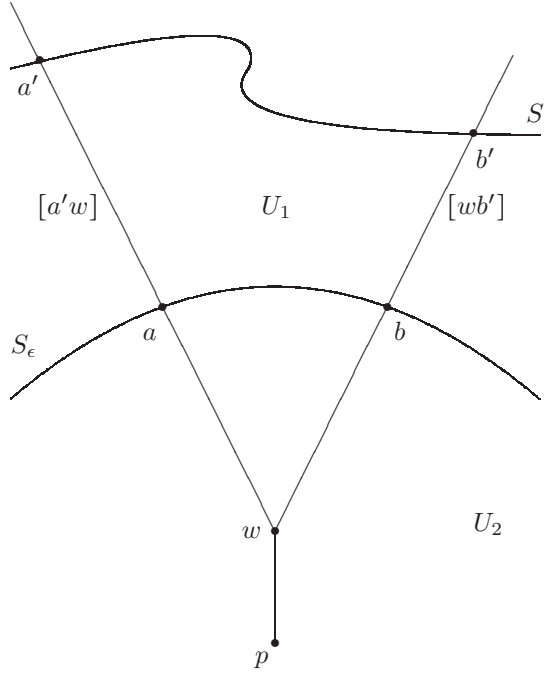

\caption{Proving $S_\epsilon$ is a circle.}\label{fig:circle}
\end{figure}

\end{center}


Now by way of contradiction, let us assume $S_\epsilon \setminus \{a, b\}$ is connected. Then
without loss of generality, $S_\epsilon \cap U_1$ must be empty. On the other hand, $U_1$ is
homeomorphic to an open disk, whose boundary is a Jordan curve (formed by the arc $(a'b')$ 
in the interior of $\mathbb D^2$, along with the portion of the boundary $S$ joining $a'$ to $b'$).
The boundary of $U_1$ contains the arc $(wa')$ passing through $a$, and the distance to $p$
varies continuously along $(wa')$ from a number $< \epsilon$ (since $w\neq a$) to a number
$> \epsilon$ (since $a \neq a'$). Pick an arc $\eta$ {\it inside} $U_1$ joining $w$ to $a'$, and 
consider the distance function restricted to $\eta$. It varies continuously from $<\epsilon$ to 
$>\epsilon$, but since $S_\epsilon \cap U_1=\emptyset$, is never equal to $\epsilon$. This is
a contradiction, completing the proof.
\end{proof}

Using the topological characterization of $\mathbb S^1$, Proposition \ref{prop:circle} now follows
immediately from Lemma \ref{lem:pc} and Lemma \ref{lem:pairs-separate}.

\

%
\section{The geometry of small distance spheres}\label{sec:rectifiable}

We now want to prove that for all $p\in \Sigma$ and $\epsilon<\epsilon_0$, $S_\epsilon$ is rectifiable, and that the lengths of these circles can be uniformly bounded above.

\begin{lem}\label{lem:finite}
Let $\Sigma$ be a complete $\textrm{CAT}(\kappa)$ surface. Then for any $z\in \Sigma$, and any $\epsilon<\epsilon_0$, $S(z,\epsilon)$ is a rectifiable curve.
\end{lem}

\begin{proof}
Fix $p$. By Proposition \ref{prop:circle}, for $\epsilon < \epsilon_0$, $S_\epsilon$ is homeomorphic to a circle.

Note that if $\epsilon'<\epsilon$, the rectifiability of $S_\epsilon$ implies that $S_{\epsilon'}$ is also rectifiable. This follows from the fact that the nearest-point projection $\pi_Z$ to a complete, convex subset $Z$ is distance non-increasing a ball of radius $<\epsilon_0$ in a $\textrm{CAT}(\kappa)$ space (see, e.g. \cite[Prop. II.2.4 (or the exercise following for $\kappa>0$)]{bh}). Applying this to the complete convex subset $Z:= B(p, \epsilon')$, and using the (global) $\textrm{CAT}(\kappa)$ geometry in $B(p,\epsilon_0)$, we see that $\pi_Z$ is just the radial projection towards $p$. In particular the image of $\pi_Z$ lies on $S_{\epsilon'}$.


\begin{center}

\setlength{\unitlength}{.5pt}

\begin{picture}(400,410)(-200,-200)

\put(0,0){\circle*{5}}
\put(0,200){\circle*{5}}
\put(0,-200){\circle*{5}}

\put(-15,-15){$p$}
\put(-55, -215){$S(1)$}
\put(-55, 175){$N(1)$}

\put(0,0){\circle{300}}

\put(0,-200){\line(0,1){400}}

\qbezier(0,200)(-190,190)(-200,0)	
\qbezier(-200,0)(-190,-190)(0,-200)	
\qbezier(0,-200)(190,-190)(200,0)	
\qbezier(200,0)(190,190)(0,200)

\put(-85,0){$S_{\epsilon/n}$}
\put(-225,50){$S_{\epsilon}$}

\put(-260,-100){$arc_W(1)$}
\put(190,-100){$arc_E(1)$}
\put(30,-40){$arc_E(n)$}

\put(0,40){\circle*{5}}
\put(0,-40){\circle*{5}}

\put(-50,50){$N(n)$}
\put(-50,-60){$S(n)$}

\put(0,0){\line(1,2){90}}
\put(90,180){\circle*{5}}
\put(0,0){\line(2,1){180}}
\put(180,90){\circle*{5}}

\put(90,190){$z^*$}
\put(190,90){$q$}

\qbezier(90,180)(20,20)(180,90)

\put(80,90){$[z^*q]$}

\end{picture}

\begin{figure}[h]
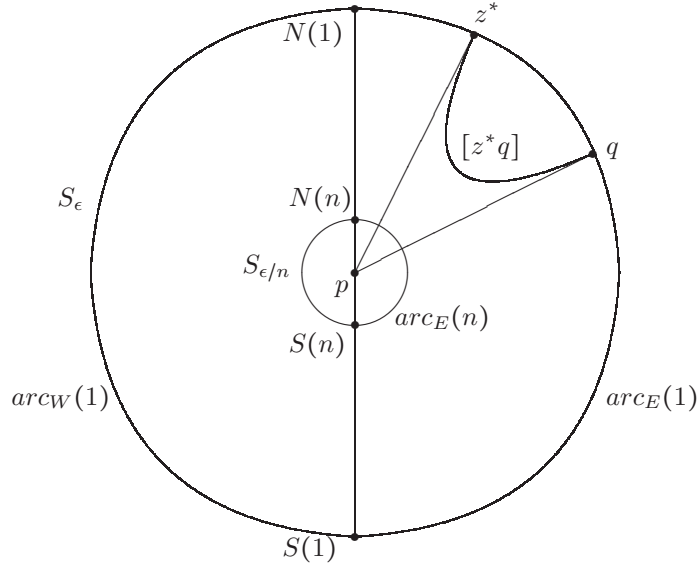

\caption{A geodesic configuration which cannot occur if $S_\epsilon$ is not rectifiable.}\label{fig:rect}
\end{figure}

\end{center}


Next, fix any geodesic $\gamma$ through $p$ and denote by $N(n)$ and $S(n)$ its two intersections with $S_{\frac{\epsilon}{n}}$ (chosen so that all $N(n)$ lie on the same component of $\gamma\setminus\{p\}$). Since $S_{\frac{\epsilon}{n}}$ is a circle, the pair $\{N(n), S(n)\}$ divides $S_{\frac{\epsilon}{n}}$ into two arcs, whose closures we call $arc_E(n)$ and $arc_W(n)$; choose these so that the relative positions of $N(n), S(n), arc_E(n)$ and $arc_W(n)$ correspond to the cardinal directions on a compass.

Note that for all $n$, at least one of $l(arc_E(n)), l(arc_W(n))$ must be infinite, since $l(S_n)$ is infinite. Without loss of generality, we may assume that $l(arc_E(n))=\infty$ for infinitely many $n$, and hence (by the remark above) for all $n$. We focus our attention now on the family $\{arc_E(n)\}_{n\in \mathbb{N}}$.

On $arc_E(1)$ define the following equivalence relation: we declare $x\sim y$ if there exists some $n$ such that the arc in $arc_E(n)$ with endpoints $[px]\cap arc_E(n)$ and $[py]\cap arc_E(n)$ is of finite length. That this is an equivalence relation is easy to check. We denote equivalence classes by $[x]$.

We note two things about $arc_E(1)/\sim$ and its equivalence classes. First, $N(1) \nsim S(1)$, for otherwise $arc_E(n)$ would have finite length for some $n$. Second, for each $x\in arc_E(1)$, $[x]$ is an interval (possibly degenerate). This is because geodesics are unique at the scale we work at, and if three points are arranged around $arc_E(1)$ in order $x<y<z$, then $[px]\cap arc_E(n)\leq [py]\cap arc_E(n)\leq [pz]\cap arc_E(n)$. Thus the decomposition of $arc_E(1)$ into the equivalence classes of $\sim$ is a decomposition into at least two disjoint subintervals (possibly degenerate) of the half-circle $arc_E(1)$.

By connectedness of $arc_E(1)$, either $[N(1)]$ or $[S(1)]$ is a singleton, or some equivalence class has a closed endpoint in the interior of $arc_E(1)$. Let $z^*$ be this endpoint or the singleton $N(1)$ or $S(1)$.

If $z^*$ is an endpoint of $arc_E(1)$, let $q$ be any other point in $arc_E(1)$. If $z^*$ is the closed endpoint of $[z^*]$ in the interior of $arc_E(1)$, let $q$ be any point in $arc_E(1)$ which lies on the $z^*$-side of $[z^*]$. We note that there are infinitely many such $q$, and by the choice of $z^*$ and the topology of the half-circle $arc_E(1)$, we may take a sequence of such $q$ approaching $z^*$. Observe that, since $q$ and $z^*$ are not equivalent, the geodesic segments $[pz^*]$ and $[pq]$ only agree at the point $p$.

Consider the geodesic segment $[z^*q]$. By the properties of geodesics in the (globally) $\textrm{CAT}(\kappa)$ set $B(p,\epsilon)$, this geodesic segment lies inside $B(p,\epsilon)$ and does not cross the geodesic $\gamma$ which divides the West and East parts of $B(p,\epsilon)$. Suppose that $[z^*q]$ does not intersect $arc_E(n)$ for some $n$ (as in Figure \ref{fig:rect}). Then the radial projection of $[z^*q]$ onto $S_n$ provides a path in $arc_E(n)$ from $[pz^*]\cap arc_E(n)$ to $[pq] \cap arc_E(n)$. Again by the distance non-increasing properties of the projection, since $[z^*q]$ has finite length, this would imply $z^*\sim q$, which contradicts the choice of $q$. Therefore the geodesic segment $[z^*q]$ must intersect $arc_E(n)$ for all $n$. It must therefore hit $p$, and by uniqueness of geodesics we conclude that $[z^*q] = [z^*p] \cup [pq]$.

Now consider $B(z^*,\epsilon)$. The work above shows that no $q$ chosen as previously described lies in $B(p^*,\epsilon)$. But this contradicts our observation above that, using the half-circle topology of $arc_E(1)$, we may take such $q$ approaching $p^*$. This contradiction concludes the proof.
\end{proof}

Let us denote by $l(\gamma)$ the length of a rectifiable curve $\gamma$. Using Lemma \ref{lem:finite} and the compactness of $\Sigma$ we have:

\begin{lem}\label{lem:uniform}
There exist $\delta_0>0$ and some uniform $C>0$ such that for all $z\in \Sigma$, $l(S(z,\delta_0))<C$.
\end{lem}

\begin{proof}

Suppose there is no uniform bound on $l(S(z,\delta_0/2))$. Then we may take a sequence of points $p_n$ in $\Sigma$ with $l(S(p_n,\delta_0/2))\geq n$. Let $p^*$ be any subsequential limit point of $(p_n)$ and note that $B(p^*, \delta_0)$ properly contains $S(p_n,\delta_0/2)$ for sufficiently large $n$. The unbounded lengths of the latter, plus again the distance non-increasing properties of nearest-point projection, would imply that the length of $S(p^*,\delta_0)$ is infinite, contradicting Lemma \ref{lem:finite}. This proves the Lemma.
\end{proof}

\
%
\section{Proof of Theorem \ref{main thm}}\label{sec:proof}

We define a particular non-expanding map from $B(p,\epsilon_0)$ to the ball of radius $\epsilon_0$ in the model space $\mathbb{H}^2$. This will be a key tool in our proof of Theorem \ref{main thm}.

Define an equivalence relation on the set of geodesic segments starting at $p$ by declaring $\gamma_1\sim\gamma_2$ if the Alexandrov angle between these segments at $p$ is 0.

\begin{defn}(See, e.g. \cite[Definition II.3.18]{bh})
The set of equivalence classes for $\sim$, equipped with the metric provided by the Alexandrov angle, is the \emph{space of directions at $p$}, denoted $S_p(\Sigma)$.
\end{defn}

The following result is standard:

\begin{prop}(See, e.g. \cite[Theorem II.3.19]{bh})\label{prop:cat(1)}
If $\tilde \Sigma$ is $\textrm{CAT}(\kappa)$ for any $\kappa$, then for each $p\in\tilde \Sigma$, the completion of $S_p(\tilde \Sigma)$ is $\textrm{CAT}(1)$.
\end{prop}

\begin{lem}\label{lem:Spcomplete}
For any $p\in \tilde \Sigma$, where $\tilde \Sigma$ is a $\textrm{CAT}(\kappa)$ surface, $S_p(\tilde \Sigma)\cong\mathbb{S}^1$.
\end{lem}

\begin{proof}
The natural projection from $S_\epsilon$ to $S_p(\tilde \Sigma)$ is continuous and, by Lemma \ref{lem:extendible}, surjective. The fiber over any point in $S_p(\tilde \Sigma)$ is easily seen to be a closed interval. Thus $S_p(\tilde \Sigma)$ is homeomorphic to a quotient of $\mathbb S^1$, where each
equivalence class is a closed interval in $\mathbb S^1$. It is a well-known result that such a quotient space
is automatically homeomorphic to $\mathbb S^1$ (this can be easily shown using the topological characterization of $\mathbb S^1$ used in the proof of Proposition \ref{prop:circle}). This establishes the Lemma.
\end{proof}

Combining Proposition \ref{prop:cat(1)} with Lemma \ref{lem:Spcomplete}, which implies that $S_p(\tilde \Sigma)$ is complete, we have

\begin{cor}\label{cor:SpCAT1}
For any $p\in \tilde \Sigma$, $S_p(\tilde \Sigma)$ is $\textrm{CAT}(1)$.
\end{cor}

We now construct the non-expanding map to the model surface $M_\kappa$ of constant curvature $\kappa$. We closely follow the proof of a similar result presented in \cite[Proposition 10.6.10]{bbi}, but for the opposite type of curvature bound (curvature bounded below, rather than above).

\begin{prop}\label{prop:contracting}
Let $\Sigma$ be a $\textrm{CAT}(\kappa)$ surface and $p$ any point in $\Sigma$. Let $\epsilon_0$ be as above. Then there is a map $f: B(p,\epsilon_0) \to M_\kappa$ such that
\begin{itemize}
	\item[(1)] ${d_{M_\kappa}}(f(x),f(y)) \leq d_\Sigma(x,y)$ for all $x,y\in B(p,\epsilon_0)$, 
	\item[(2)] ${d_{M_\kappa}}(f(p),f(y)) = d_\Sigma(p,y)$ for all $y\in B(p,\epsilon_0)$, and 
	\item[(3)] $f(B(p,\epsilon)) = B_{M_\kappa}(f(p),\epsilon)$.
\end{itemize}
\end{prop}

\begin{proof}
By the choice of $\epsilon_0$, we can work in $\Sigma$ or lift $B(p,\epsilon_0)$ homeomorphically to $\tilde \Sigma$. By corollary \ref{cor:SpCAT1}, $S_p(\Sigma)$ is $\textrm{CAT}(1)$. It is homeomorphic to $\mathbb{S}^1$, so it is easy to see that there is a surjective map $g:S_p(\Sigma)\to \mathbb{S}^1$ which is non-expanding:
\[ d_{\mathbb{S}^1}(g(v),g(w)) \leq \measuredangle_p(v,w) \quad \mbox{ for all } \quad v,w\in S_p(\tilde \Sigma). \] 

Let $K^{\kappa}_p(\tilde \Sigma)$ denote the $\kappa$-cone over $S_p(\tilde \Sigma)$ (see, e.g. \cite[\S 10.2.1]{bbi}). By its definition, and the fact that $S_p(\Sigma)$ is $\textrm{CAT}(1)$, $K^{\kappa}_p(\tilde \Sigma)$ is $\textrm{CAT}(\kappa)$ (\cite[Theorem 4.7.1]{bbi}). On the ball $B(p,\epsilon_0)$ define a logarithm map defined as follows:
\[ \log_p: B(p, \epsilon_0) \to K^{\kappa}_p(\tilde \Sigma) \]
\[ p \mapsto \mbox{origin}=o \]
\[ x\neq p \mapsto (v,d_\Sigma(p,x)) \]
where $v$ is the direction in $S_p(\tilde \Sigma)$ of the geodesic segment $[px]$. From the non-expanding property of $g$ and the definition of $K^{\kappa}_p(\tilde \Sigma)$, 
\[ d_{K^{\kappa}_p(\tilde \Sigma)}(\log_p(x),\log_p(y)) \leq d_\Sigma(x,y) \quad \mbox{ for all } \quad x,y\in B(p,\epsilon_0). \]
By its definition, $\log_p$ preserves distance from the origin, and by its definition, $\log_p$ maps $B(p,\epsilon)$ surjectively onto $B_{K^{-1}_p(\tilde \Sigma)}(o, \epsilon).$

Again, using the definition of $K^{\kappa}_p(\tilde \Sigma)$, the non-expanding map $g:S_p(\tilde \Sigma) \to \mathbb{S}^1$ extends to a map $G:K^{\kappa}_p(\tilde \Sigma)\to M_\kappa$, obtained by realizing $M_\kappa$ as the $\kappa$-cone over $\mathbb{S}^1$. The map is non-expanding since $g$ is, and preserves distance from the origin. $G$ sends $B_{K^{\kappa}_p(\tilde \Sigma)}(o,\epsilon)$ surjectively to $B_{M_\kappa}(G(o),\epsilon)$ because $g$ is surjective. Then $f=G\circ \log_p$ is the desired map.
\end{proof}

We are now ready to prove Theorem \ref{main thm}.

\begin{proof}[Proof of Theorem \ref{main thm}]
Let $\delta_0=\epsilon_0$ and let $\delta<\delta_0$. First we bound $H^2(B(p,\delta))$ below.

Let $f$ be the non-expanding map provided by Proposition \ref{prop:contracting}. Since $f$ preserves radial distance from the origin, $f(B(p,\delta)) \subseteq B(f(p),\delta)$. Fix any $\rho>0$ and suppose $\{U_i\}$ is a countable cover of $B(p,\delta)$ with $\diam(U_i)<\rho$. Then the collection $\{ f(U_i)\}$ covers $B(f(p),\delta)$ and $\diam(f(U_i))\leq \diam(U_i) <\rho$ as $f$ is non-expanding. Therefore, 
\[ \Sigma_{i} \diam{U_i}^2 \geq \Sigma_i \diam{f(U_i)}^2 \geq H^2_\rho(B(f(p),\delta)). \]
The right-hand quantity approaches $H^2(B(f(p),\delta))$ as $\rho\to 0$. But this is just the volume of a $\delta$ ball in $M_\kappa$, the model space of curvature $\kappa$. Since this is independent of the choice of basepoint $p$, we obtain the desired uniform lower bound on $H^2(B(p,\delta))$.

\

Now we bound $H^2(B(p,\delta))$ above. This portion of the proof uses the bound on the length of $S_\delta$ obtained in Lemma \ref{lem:uniform}. It is sufficient to bound $H^2(B(p,\delta_0))$ uniformly above.

Fix $\rho<\delta_0$. Let $E_\rho$ be any finite $\frac{\rho}{2}$-spanning subset of $S_{\delta_0}$. The circumference bound allows us to uniformly bound $\#E_\rho$. Index $x_j\in E_\rho$ in order around $S_{\delta_0}$. Let $T_j$ be the geodesic triangle with vertices $p, x_j, x_{j+1}$. $T_j$ has edges of length $\delta_0, \delta_0$ and $<\rho$. Let $\tau_j$ be the corresponding comparison triangle in $M_\kappa$ (see Figure \ref{fig:compare}).


\begin{center}

\setlength{\unitlength}{.5pt}

\begin{picture}(600,400)(-300,-150)

\put(-210, 180){$\Sigma$}
\put(190, 180){$M_\kappa^2$}

\qbezier(-200,150)(-340,140)(-350,0)	
\qbezier(-350,0)(-340,-140)(-200,-150)	
\qbezier(-200,-150)(-40,-140)(-50,0)	
\qbezier(-50,0)(-60,140)(-200,150)	

\put(-370,120){$B(p,\delta_0)$}

\put(-200,0){\circle*{5}}
\put(-215,5){$p$}

\put(-200,0){\line(1,1){108}}
\put(-200,0){\line(2,1){135}}

\put(-92,108){\circle*{5}}
\put(-65,67){\circle*{5}}
\put(-52,20){\circle*{5}}
\put(-50,-20){\circle*{5}}
\put(-58,-70){\circle*{5}}
\put(-80,-105){\circle*{5}}
\put(-115,-131){\circle*{5}}
\put(-170,-147){\circle*{5}}

\put(-230,-146){\circle*{5}}
\put(-285,-125){\circle*{5}}
\put(-320,-94){\circle*{5}}
\put(-342,-48){\circle*{5}}
\put(-350,0){\circle*{5}}
\put(-340,50){\circle*{5}}
\put(-308,108){\circle*{5}}

\put(-260,137){\circle*{5}}
\put(-210,150){\circle*{5}}
\put(-145,140){\circle*{5}}

\put(-88,108){$x_{j+1}$}
\put(-61,67){$x_j$}

\qbezier(-92,108)(-120,60)(-65,67)

\put(-150,80){$T_j$}

\put(70,-20){\circle*{5}}
\put(270,80){\circle*{5}}
\put(300,-20){\circle*{5}}

\multiput(70,-20)(60,0){4}{\circle*{5}}
\multiput(95,-8)(50,25){4}{\circle*{5}}

\qbezier(95,-8)(95,-15)(130,-20)
\qbezier(145,18)(125,1)(130,-20)
\qbezier(145,18)(145,-5)(190,-20)
\qbezier(195,43)(185,10)(190,-20)
\qbezier(195,43)(185,0)(250,-20)
\qbezier(245,68)(225,10)(250,-20)
\qbezier(245,68)(245,0)(300,-20)

\put(190,-40){$\bar y_i$}
\put(190,54){$\bar z_i$}

\put(60,-10){$\bar p$}
\put(275,80){$\bar x_{j+1}$}
\put(305,-20){$\bar x_j$}

\put(90,60){$\tau_j$}

\put(70,-20){\line(2,1){200}}
\put(70,-20){\line(1,0){230}}

\qbezier(270,80)(250,10)(300,-20)

\put(-280,-55){$E_\epsilon$}
\put(-285,-60){\vector(-1,-1){30}}
\put(-290,-50){\vector(-1,0){40}}

\end{picture}

\begin{figure}[h]
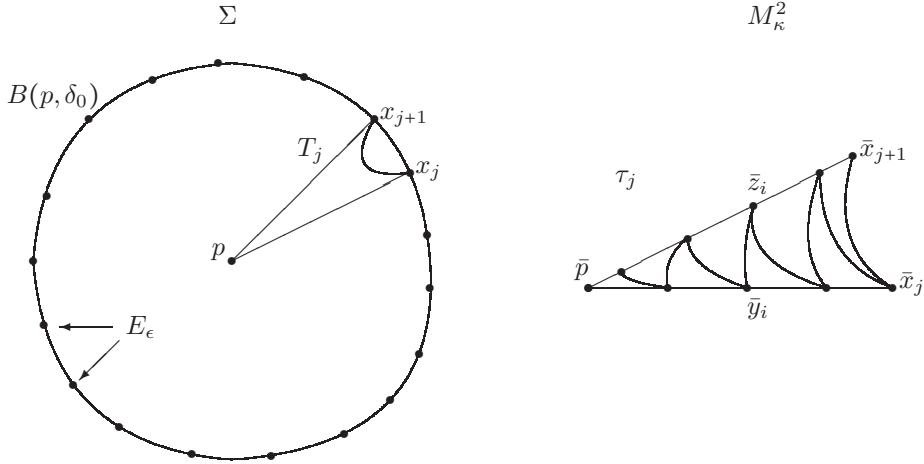

\caption{Comparison triangles for the upper bound.}\label{fig:compare}
\end{figure}

\end{center}


In $M_\kappa$, let $r_j$ be the $\delta_0$-length edge from $\bar p$ to $\bar x_j$. Let $r(\rho)$ be the number of $\rho$-balls centered at points on $r_j$ necessary to cover $\tau_j$ with centers an $\rho$-spanning set in $r_j$. Let their centers be $\bar y_1, \ldots \bar y_{r(\rho)}$. Note that we can take $r(\rho) = C^\prime \frac{\delta_0}{\rho}$ for $C^\prime$ a constant independent of $\rho$.  Now, pick $\bar z_i$ on the other $\delta_0$-length side of $\tau_j$ and in $B(\bar y_i,\rho)\cap B(\bar y_{i+1},\rho)$. Draw in $\tau_j$ and in $T_j$ the zig-zagging segments connecting $y_i$ to $z_i$ to $y_{i+1}$ to $z_{i+1}$ etc. These partition $\tau_j$ and $T_j$ into a union of triangles. In $T_j$ each has all three sides of length $<\rho$, by the choice of $\bar y_i$ and $\bar z_i$. By $\textrm{CAT}(\kappa)$, the corresponding triangles in $T_j$ also have all sides of length $<\rho$, and then, again by $\textrm{CAT}(\kappa)$, we see that the $\rho$-balls centered at $y_i$ in $T_j$ cover $T_j$.

Since $l_g(S_{\delta_0})<C$, we can take $\#E_\rho \leq \frac{2C}{\rho}$. Then $B(p,\delta_0)$ can be covered by $2C C^\prime \frac{\delta_0}{\rho^2}$ balls of radius $\rho$. Thus we obtain a finite cover $\{U_i\}$ of $B(p,\delta_0)$ with $\diam(U_i)\leq 2\rho$ satisfying, 
\[ \sum_i \diam(U_i)^2 \leq 2C C^\prime \frac{\delta_0}{\rho^2}(2\rho)^2 = 8CC^\prime \delta_0 \]
which is bounded above independently of $\rho$. It follows that $H^2(B(p,\delta_0))\leq 8CC^\prime \delta_0$, which verifies that $H^2(B(p,\delta_0))$ is finite and bounded above uniformly in $p$.

\end{proof}

%

\section{Local geometry of $\textrm{CAT}(-1)$ spaces in higher dimensions}\label{sec:example}

In this section we make a few remarks on the obstructions to extending the local geometry results we proved for surfaces to higher dimensions. We note that the proofs in the previous section rely heavily on the 2-dimensionality of $\Sigma$. We do not know if an analogue of Theorem \ref{main thm} holds for $\textrm{CAT}(\kappa)$ metrics on closed higher dimensional manifolds. One of the first steps in our proof was Proposition \ref{prop:circle}, which showed that the small enough metric spheres inside locally $\textrm{CAT}(\kappa)$ surfaces were homeomorphic to the circle $\mathbb S^1$. The analogous statement {\it fails} in dimensions $\geq 5$, as the well-known example below shows.

\begin{prop}
For each dimension $n\geq 5$, there exists a closed $n$-manifold $M$ equipped with a piecewise hyperbolic, locally $\textrm{CAT}(-1)$ metric, and a point $p\in M$ with the property that for all small enough $\epsilon$, the $\epsilon$-sphere $S_\epsilon$ centered at $p$ is {\bf not} homeomorphic to $\mathbb S^{n-1}$. In fact, $S_\epsilon$ is not even a manifold.
\end{prop}

\begin{proof}

Such examples can be found in the work of Davis and Januszkiewicz \cite[Theorem 5b.1]{mike-tadek}. We briefly summarize the construction for the convenience of the reader. Start with a closed smooth homology sphere $N^{n-2}$ which is not homeomorphic to $\mathbb S^{n-2}$. Such manifolds exist for all $n\geq 5$, and are quotients of $\mathbb S^{n-2}$ by a suitable perfect group $\pi_1(N^{n-2})$. Take a smooth triangulation of $N^{n-2}$, and consider the induced triangulation $\mathcal T$ on the double suspension $\Sigma ^2 (N^{n-2})$. By work of Cannon and Edwards, $\Sigma ^2 (N^{n-2})$ is homeomorphic to $\mathbb S^n$. The triangulation $\mathcal T$ on $\mathbb S^n$ is not a PL-triangulation, as there exists a $4$-cycle in the $1$-skeleton of the triangulation whose link is homeomorphic to $N^{n-2}$. Now apply the strict hyperbolization procedure of Charney and Davis \cite{charney-davis} to the triangulated manifold $(\mathbb S^n, \mathcal T)$.

This outputs a piecewise-hyperbolic, locally CAT(-1) space $M$. A key point of the hyperbolization procedure is that it preserves the local structure. Since the input $(\mathbb S^n, \mathcal T)$ is a closed $n$-manifold, the output $M$ is also a closed $n$-manifold. The $4$-cycle in $\mathcal T$ whose link was homeomorphic to $N^{n-2}$ now produces a closed geodesic $\gamma$ in $M$, whose link is still homeomorphic to $N^{n-2}$ (i.e. the ``unit normal'' to $\gamma$ forms a copy of $N^{n-2}$). It follows from this that, picking the point $p$ on $\gamma$, all small $\epsilon$-spheres $S_\epsilon$ are homeomorphic to the suspension $\Sigma N^{n-2}$. Since $N^{n-2}$ was {\bf not} the standard sphere, $S_\epsilon \cong \Sigma N^{n-2}$ fails to be a manifold at the suspension point $x$, as every small punctured neighborhood of $x$ will have non-trivial $\pi_1$. We refer the reader to \cite[Section 5]{mike-tadek} for more details.

\end{proof}

This cautionary example suggests that small metric spheres in high-dimensional locally CAT($\kappa$) manifolds could exhibit pathologies. In view of these results, and the interest in obtaining higher dimensional analogs, we raise the following question.

\begin{ques}
Let $M$ be a closed $n$-manifold equipped with a locally $\textrm{CAT}(-1)$ metric of Hausdorff dimension $d$. Can $d$ ever be strictly larger than $n$? Do the uniform bound conditions of Theorem \ref{main thm} hold in higher dimensions? 
\end{ques}

The authors suspect that examples with $d>n$ do indeed exist in higher dimensions.

%

\section{Entropy rigidity in CAT(-1)}\label{sec:entropy}

In this section we present an entropy rigidity result for closed $\textrm{CAT}(-1)$ manifolds. This result generalizes Hamenst\"adt's entropy rigidity result from \cite{hamenstadt_entropy} to the $CAT(-1)$ setting. It is very closely related to, and in fact relies on, a rigidity result of Bourdon. The main addition to Bourdon's theorem is the connection to topological entropy via a theorem of Leuzinger (generalizing work of Manning). The theorem draws heavily on the work of others, but we have not seen it presented in this form in the literature.

We remark that our work on Hausdorff 2-measure for surfaces, presented in the earlier sections of this paper, was inspired in part by condition ($\star$) that features in Leuzinger's theorem below (Theorem \ref{leuzinger}). 

Recall that a negatively curved locally symmetric space has universal cover isometric to $\mathbb{H}^m_\mathbf{K}$, where $m\geq 2$ and $\mathbf{K}$ is $\mathbb{R}, \mathbb{C}$, the quaternions $\mathbf{H}$, or the octonions $\mathbf{O}$ (with $m=2$). We suppose that the metrics on these spaces are scaled so that the maximum sectional curvature is $-1$. The topological entropy for the geodesic flow on a compact locally symmetric space modeled on $\mathbb{H}^m_\mathbf{K}$ is $km+k-2$ where $k=\dim_\mathbb{R}\mathbf{K}$. (For the definition of topological entropy, see \cite[\S 3.1]{kh}.) 

More generally, if one has a locally $\textrm{CAT}(-1)$ metric $d$ on the manifold $M$, 
there is an associated 
space $\mathcal G(M,d)$ of geodesics in $M$: this is the space of locally isometric maps $\mathbb R \rightarrow M$. Via lifting to the universal cover, this space is topologically a quotient of
$\mathbb S^{n-1} \times \mathbb S^{n-1} \times \mathbb R$ by a suitable $\Gamma:= \pi_1(M)$ action. There is a 
natural flow $\phi^d_t$ on $\mathcal G(M,d)$, given by precomposition with an $\mathbb R$-translation. This is called the
geodesic flow associated to the metric $d$. One can again measure the topological entropy of this flow. 

The main result of this section is the following restatement of our Theorem \ref{entropy rigidity}:

\begin{thm}\label{thm:rigidity}
Let $(X,d)$ be a closed $n$-dimensional manifold equipped with a locally $\textrm{CAT}(-1)$ metric $d$ (not necessarily Riemannian). Suppose that $X$ also supports a locally symmetric Riemannian metric $g$, under which $(X,g)$ is locally modeled on $\mathbb{H}^{m}_\mathbf{K}$, normalized to have maximum sectional curvature $-1$. Then 
\[ h_{top}(\phi^d_t)\geq h_{top}(\phi^g_t) = km+k-2\]
and if equality holds in the above, $(X,d)$ is also locally symmetric. If $n>2$, $(X,d)$ and $(X,g)$ are isometric.
\end{thm}

This should be compared with the main theorem of \cite{hamenstadt_entropy}, which establishes
this same result when the metric $d$ is a Riemannian metric with sectional curvature $\leq -1$.
The key element in this proof is the following theorem of Bourdon, which he notes is a generalization of Hamenst\"adt's work.

\begin{thm}\cite[Th\'eor\`eme 0.3 and following remarks]{bourdon}\label{bourdon}
Let $\tilde X$ be a $\textrm{CAT}(-1)$ space with a cocompact isometric action by $\Gamma$ which also acts convex cocompactly by isometries on a negatively curved symmetric space $S=\mathbb{H}^{m}_\mathbf{K}$. Let $\Lambda$ be the limit set in $\partial^\infty \tilde X$ of $\Gamma$, and let $d^*$ and $g^*$ denote the visual metrics on $\partial^\infty \tilde X$ and $\partial^\infty S$, respectively. Suppose that $\dim_H(\Lambda, d^*) = \dim_H(\partial^\infty S, g^*)$. Then there exists an isometric embedding $F$ from $S$ into $\tilde X$ such that $\partial^\infty F(\partial^\infty S) = \Lambda$. If $\dim_\mathbb{R}S>2$, then this embedding is $\Gamma$-equivariant.
\end{thm}

This theorem proves rigidity for the extremal value of $\dim_H(\Lambda, d^*)$ -- an extremal value which had already been computed by Pansu:

\begin{thm}\cite[Th\'eor\`eme 5.5]{pansu_confdim}\label{pansu}
With notation as in Theorem \ref{bourdon},
\[ \dim_H(\Lambda,d^*) \geq \dim_H(\partial^\infty S,g^*) = km+k-2.\]
\end{thm}

A definition for the visual metrics referenced above can be found in \cite[\S III.H.3]{bh}. Here we give a short definition of a distance function which is Lipschitz equivalent to any visual metric. Since Hausdorff dimension can be calculated using any distance function in the Lipschitz equivalence class, this suffices for our purposes.

\begin{defn}
Let $(\tilde X, \tilde d)$ be a $\textrm{CAT}(-1)$ metric space and let $\zeta, \eta \in \partial^\infty \tilde X$. Fix some basepoint $x\in\tilde X$ and define
\[ d^*_x (\zeta,\eta) = e^{-\tilde d(x,[\zeta\eta])},\]
where $[\zeta\eta]$ denotes the bi-infinite geodesic in $\tilde X$ with endpoints at $\zeta$ and $\eta$.
It is straightforward to see that the Lipschitz class of $d_x^*$ is independent of $x$.
\end{defn}

We also need the following result of Manning, as generalized to the $\textrm{CAT}(-1)$ setting by Leuzinger.

\begin{defn}
Let $(X,d)$ be a closed manifold with some metric $d$ and endowed with a measure $vol$. Let $(\tilde X, \tilde d)$ denote its universal cover. Then the \emph{volume growth entropy} of $(\tilde X,\tilde d)$ with respect to $vol$ is
\[ h_{vol}(\tilde X,\tilde d):= \limsup_{R\to\infty}\frac{1}{R} vol(B(p,R))\]
where $B(p,R)$ is the ball of radius $R$ about a point $p$ in $\tilde X$.
\end{defn}

Manning shows that this is independent of the choice of $p$ and that, for Riemannian manifolds, the $\limsup$ is in fact a limit.

\begin{thm}[Leuzinger, \cite{leuzinger}; compare with Manning \cite{manning}]\label{leuzinger}
Let $(\tilde X,\tilde d)$ be a geodesically complete, locally geodesic metric space endowed with a measure $vol$, and having compact quotient $\tilde X/\Gamma$. Assume that
\begin{itemize}
	\item[($\star$)] there exists some $0<\delta_0<\infty$ such that for all $0<\delta\leq \delta_0$,
	\[ 0<\inf_{z\in X}vol(B(z,\delta)) \mbox{ and } \sup_{z\in X}vol(B(z,\delta))<\infty. \]
\end{itemize}
Then
\[h_{top}(\phi_t)\geq h_{vol}(\tilde X,\tilde d).\]
If, in addition, $(\tilde X, \tilde d)$ is $\textrm{CAT}(0)$, then
\[ h_{top}(\phi_t) = h_{vol}(\tilde X,\tilde d). \]
\end{thm}

Manning proved this result for Riemannian manifolds (for which condition ($\star$) holds automatically); Leuzinger extracts the key conditions from that proof, namely ($\star$) and convexity of the function $d(c_1(t),c_2(t))$, the distance between a pair of geodesics. The latter is automatically satisfied in a $\textrm{CAT}(0)$ space. 

In the surface case, the validity of condition ($\star$) for $2$-dimensional Hausdorff measure is established in our Theorem \ref{main thm}, so applying Theorem \ref{leuzinger} immediately yields Theorem \ref{thm:surface vol entropy}.

Finally, we note the following.

\begin{prop}\label{prop:dim=vol}
Let $(X,d)$ be compact and locally $\textrm{CAT}(-1)$, and suppose that $vol$ is a measure on $X$ giving finite, nonzero measure to $X$. Then
\[ h_{vol}(\tilde X,\tilde d) = \dim_H(\partial^\infty\tilde X, d_x^*)\]
where $\dim_H$ denotes Hausdorff dimension.
\end{prop}

\begin{proof}[Proof of Prop \ref{prop:dim=vol}]
Let $N(R) = \#\{ \gamma \in \Gamma : d_{\tilde X}(*, \gamma *) \leq R\}$. The result follows easily from the following facts. First, due to Bourdon \cite[Theorem 2.7.5]{bourdon}, we have that
\[ \dim_H(\partial^\infty\tilde X, d_x^*) = \limsup_{R\to \infty} \frac{1}{R} \log N(R).\]
Second, if we let $F$ be a bounded fundamental domain for the $\Gamma$-action, with volume $V$ and diameter $D$,
\[ vol(B(*,R)) \leq V\cdot N(R+D)\]
and 
\[ V\cdot N(R) \leq vol(B(*,R+D)).\] 
\end{proof}

We now prove Theorem \ref{thm:rigidity}.

\begin{proof}[Proof of Theorem \ref{thm:rigidity}]

Since $X$ is a manifold, we may fix some Riemannian metric on $X$, for example the locally symmetric metric $g$ (though any metric will do). This defines a Riemannian volume form $vol$ on $X$. We first claim that $vol$ satisfies condition ($\star$) from Theorem \ref{leuzinger}.

Indeed, take $\delta_0$ so small that any $\delta_0$-ball for the $d$-metric on $X$ lifts isometrically to $(\tilde X, \tilde d)$. Then the volume of such a ball is clearly bounded above by the Riemannian volume of $X$, which is finite as $X$ is compact. This establishes the uniform upper bound on $vol(B^d(z,\delta))$.

Suppose that there is no uniform positive lower bound. Then for any fixed $\delta>0$, there exists a sequence of points $z_n$ in $X$ such that $vol(B^d(z_n,\delta)) \to 0$. By compactness we may extract a convergent subsequence $z_{n_i}$ with limit $z^*$. For sufficiently large $i$, $B^d(z^*,\delta/2)\subset B^d(z_{n_i},\delta)$ and hence $vol(B^d(z^*,\delta/2))=0$. Since $d$ and the Riemannian metric $g$ induce the same topology on $X$, there exists some $\epsilon>0$ such that $B^g(z^*,\epsilon)\subset B^d(z^*,\delta/2)$. But the Riemannian volume of $B^g(z^*,\epsilon)$ must be strictly positive, giving a contradiction. Therefore, condition ($\star$) holds for $vol$.

Let $(S,g)$ be the negatively curved symmetric space on which $(X,g)$ is locally modelled.
Using Theorem \ref{leuzinger}, Proposition \ref{prop:dim=vol} and Theorem \ref{pansu} and the fact that $\Gamma$ acts cocompactly on $X$ (and so $\Lambda=\partial^\infty \tilde X$), we have 
\begin{align}
	h_{top}(\phi^d_t) &= h_{vol}(\tilde X,\tilde d) \nonumber \\
				&=\dim_H(\partial^\infty\tilde X, d_x^*) \nonumber \\
				& \geq \dim_H(\partial^\infty S,g_x^*) \nonumber \\
				& = kn'+k-2. \nonumber
\end{align}

Now suppose that $h_{top}(\phi^d_t)$ achieves the lower bound. Then by Theorem \ref{bourdon}, there exists an isometric embedding $F$ from $S$ into $X$ with $\partial^\infty F(\partial^\infty S)=\Lambda =\partial^\infty \tilde X$; if $n>2$, this embedding can be taken to be $\Gamma$-equivariant. Since $X$ is a manifold with $\dim(X)=\dim(S)$ and $\partial^\infty F(\partial^\infty S)$ is the full boundary at infinity of $\tilde X$, it follows that this isometric embedding is in fact surjective. This finishes the proof.

\end{proof}

\bibliographystyle{alpha}
\bibliography{biblio}

\end{document}